\documentclass[10pt,reqno]{amsart}

\usepackage{amssymb,amsrefs,amsmath,mathtools}
\usepackage{a4wide,graphicx,nicefrac}
\usepackage[shortlabels]{enumitem}
\usepackage{mathrsfs}
\usepackage{hyperref}
\usepackage{ourbib}
\usepackage{xcolor}
\usepackage{stmaryrd}
\usepackage{mathptmx}
\usepackage{todonotes}

\setlist{leftmargin=4mm}
\newtagform{simple}{}{}{}

\hypersetup{
    colorlinks,
    linkcolor={red!50!black},
    citecolor={blue!50!black},
    urlcolor={blue!80!black}
}

\newcommand{\Z}{\mathbb{Z}}
\newcommand{\R}{\mathbb{R}}
\renewcommand{\sf}{\mathsf{sf}}
\newcommand{\APS}{\mathsf{APS}}
\newcommand{\mAPS}{\mathsf{mAPS}}
\newcommand{\UN}{\mathsf{U}(N)}
\newcommand{\deR}{\mathrm{dRh}}
\newcommand{\N}{\mathbb{N}}
\newcommand{\C}{\mathbb{C}}
\newcommand{\D}{\mathscr{D}}
\newcommand{\vol}{\mathsf{vol}}
\newcommand{\scal}{\mathrm{scal}}
\newcommand{\op}{\mathrm{op}}
\newcommand{\nbR}{\nabla^{\R^{n+1}}}
\newcommand{\nbS}{\nabla^{S}}

\DeclareMathOperator{\ind}{\mathrm{ind}}
\DeclareMathOperator{\id}{\mathrm{id}}
\DeclareMathOperator{\cs}{\mathsf{ch}^{\mathrm{odd}}}
\DeclareMathOperator{\ch}{\mathsf{ch}}
\DeclareMathOperator{\tr}{\mathsf{tr}} 
 
\DeclareMathOperator{\Ahat}{\mathsf{\hat A}}
\DeclareMathOperator{\sign}{\mathsf{sign}}
\DeclareMathOperator{\rk}{\mathsf{rk}}

\newtheorem*{theorem}{Theorem}
\newtheorem{lemma}{Lemma}
\newtheorem{corollary}{Corollary} 

\theoremstyle{definition}
\newtheorem*{remark}{Remark}
 
\newtheorem*{example}{Example} 
\newtheorem{numberedexample}{Example} 

\allowdisplaybreaks

\begin{document}

\title{Spectral flow and the Atiyah-Patodi-Singer index theorem} 
\author{Christian B\"ar}
\address{Universit\"at Potsdam, Institut f\"ur Mathematik, 14476 Potsdam, Germany}
\email{\href{mailto:christian.baer@uni-potsdam.de}{christian.baer@uni-potsdam.de}}
\urladdr{\url{https://www.math.uni-potsdam.de/baer/}}
\author{Remo Ziemke}
\email{\href{mailto:remo.ziemke@uni-potsdam.de}{remo.ziemke@uni-potsdam.de}}
\urladdr{\url{https://www.math.uni-potsdam.de/professuren/geometrie/personen/remo-ziemke/}}

\begin{abstract} 
We establish a formula for the spectral flow of a smooth family of twisted Dirac operators on a closed odd-dimensional Riemannian spin manifold, generalizing a result by Getzler.
The spectral flow is expressed in terms of the $\Ahat$-form of the manifold, the odd Chern character form of the family of connections, and the $\xi$-invariants of the initial and final operators.
Our proof is based on a reduction to the Atiyah-Patodi-Singer index theorem for manifolds with boundary, which provides a conceptually very simple approach to the problem.
As an application, we give a proof of Llarull's rigidity theorem for scalar curvature of strictly convex hypersurfaces in Euclidean space which works the same in even and odd dimensions.
\end{abstract}

\keywords{Dirac operator, spectral flow, Atiyah-Patodi-Singer index theorem, odd Chern character, $\eta$-invariant, $\xi$-invariant, Llarull's theorem}

\subjclass[2020]{53C27, 58J20, 58J28, 58J30}

\date{\today}

\maketitle

\section{Introduction} 

Let $M$ be a closed Riemannian spin manifold.
Then the Dirac operator on $M$ is defined and its Fredholm index can be computed in geometric terms using the Atiyah-Singer index theorem.
This has seen numerous applications in geometry and topology.
For example, in scalar curvature geometry one can rule out the existence of metrics with positive scalar curvature in a lot of cases and one can prove Llarull's rigidity result for the round sphere \cites{Ll} in even dimensions.
In these applications, the index theorem is typically used to show existence of nontrivial harmonic spinors which are then a main tool in the arguments. 

In odd dimensions, the index of the Dirac operator vanishes so that the index theorem usually cannot be applied directly.
One can then try to reduce to the even-dimensional case using some geometric construction as in Llarull \cites{Ll}.
More conceptually, one can consider families of Dirac operators instead of a single operator and use its spectral flow rather than an index.
If the spectral flow in nonzero, then again we must have a nontrivial harmonic spinor, at least for one of the operators in the family.
For example, this provides a very natural proof of Llarull's theorem in odd dimensions, as shown by Li, Su, and Wang in \cite{LSW}.
We illustrate this in some detail in the last section.
Similarly, Gromov's long neck problem was solved in even dimensions by Cecchini and Zeidler in \cite{CZ} using index theory, while the odd-dimensional case was treated by Shi in \cite{Shi2} using spectral flow.
Furthermore, the fact that infinite $K$-cowaist is an obstruction to positive scalar curvature is classically proved in even dimensions by index theory and recently in odd dimensions by Shi in \cite{Shi1} using spectral flow.
Finally, there are even applications of spectral flow to topological insulators, see e.g.\ \cite{YWX}.

In our setup, the family of Dirac operators arises from a family of metric connections on an arbitrary fixed Hermitian vector bundle over $M$.
We twist the spinorial Dirac operator with these connections and consider the resulting family of twisted Dirac operators.
Getzler proved a formula for the spectral flow of such a family in \cite{G} in case the connections arise from a map of $M\to \UN$ as in our Example~\ref{ex.UN}.
The geometric expression looks remarkably like the Atiyah-Singer index formula in even dimensions.
We simply have to replace the usual (even) Chern character form with an odd version of it.

In this note, we prove  Getzler's result for arbitrary families of metric connections.
In particular, the twist bundle need not be trivial.
More importantly, our proof is much simpler because we reduce the problem to the Atiyah-Patodi-Singer index formula for manifolds with boundary \cite{APS1}.
From an abstract functional analytic point of view, this reduction is well known, see e.g.\ the work of van den Dungen and Ronge \cite{Ronge}.
In fact, we will use their result in Lemma~\ref{lem.sf=ind}.
The relation between spectral flow and the Atiyah-Patodi-Singer index has also been investigated by Ebert in \cite{Ebert} from a $K$-theoretic point of view but this aspect is not needed here.

The main result of this note is the following theorem.

\begin{theorem}
Let $M$ be a closed Riemannian spin manifold of odd dimension and $E \to M$ a hermitian vector bundle. 
Let $\nabla^\bullet=(\nabla^s)_{s \in [a,b]}$ be a smooth 1-parameter family of metric connections on $E$ and $D^s$ the induced twisted Dirac operators acting on sections of $\Sigma M \otimes E$.

Then the spectral flow of the family $D^\bullet=(D^s)_{s \in [a,b]}$ is given by
$$
\sf (D^\bullet) = -\int_M \Ahat(TM) \wedge \cs(\nabla^\bullet) + \xi(D^b) - \xi(D^a).
$$
\end{theorem}

Here $\Ahat(TM)$ denotes the $\Ahat$-form of the tangent bundle with respect to the Levi-Civita connection and $\cs(\nabla^\bullet)$ the odd Chern character form of the family of connections.
Moreover, $\xi(D)$ denotes the $\xi$-invariant of the operator $D$, see the next section for details.

Since the $\xi$-invariant depends only on the spectrum of the operator, we obtain the following corollary.

\begin{corollary}
If, in addition to the assumptions in the theorem, the operators $D^a$ and $D^b$ are isospectral (e.g.\ if $D^a=D^b$), then 
$$
\sf (D^\bullet) = -\int_M \Ahat(TM) \wedge \cs(\nabla^\bullet).
$$
\end{corollary}

This note is organized as follows:
In Section~\ref{sec.Setup} we recall the necessary concepts and tools, including the $\eta$- and $\xi$-invariants, characteristic forms, the Atiyah-Patodi-Singer index theorem, and spectral flow.
In Section~\ref{sec.Proof} we carry out the proof of the theorem.
In the first example of the final section we discuss Getzler's setup.
The second example concerns a specific family of connections on the spinor bundle of the sphere.
We illustrate its usefulness by giving a streamlined proof of Llarull's rigidity theorem for strictly convex hypersurfaces of Euclidean space in the spirit of \cite{LSW}.

\medskip

\textit{Acknowledgments:} 
The authors thank Oskar Riedler for helpful discussions.
CB was partially supported by the DFG through the SPP 2026 \emph{Geometry at Infinity}.
RZ was partially funded by the EU through the ERC Starting Grant 101116001 \emph{COMSCAL}. 
Views and opinions expressed are however those of the authors only and do not necessarily reflect those of the European Union or the European Research Council. 
Neither the European Union nor the granting authority can be held responsible for them.

\section{Setup and preliminaries}
\label{sec.Setup}

We describe the setup and provide the necessary background.

\subsection{\texorpdfstring{The $\eta$- and $\xi$-invariants}{The η- and ξ-invariants}}
We start with the $\eta$-invariant.
Let $D$ be a self-adjoint elliptic differential operator of positive order acting on sections of a Hermitian vector bundle over a closed Riemannian manifold.
Then $D$ has discrete real spectrum consisting of eigenvalues of finite multiplicity only.
The $\eta$-function of $D$ is defined for $z\in\C$ by
$$
\eta(D,z) := \sum_{\lambda \neq0} \sign(\lambda) |\lambda|^{-z}.
$$
Here the summation is over all nonzero eigenvalues $\lambda$ of $D$ repeated according to multiplicity.
It follows from the Weyl asymptotics that this series converges for $\Re(z)$ sufficiently large and defines a holomorphic function on the corresponding half-plane in $\C$.
The $\eta$-function extends to a meromorphic function on all of $\C$ and does not have a pole at $z=0$, see e.g.\ Theorem~4.3.8 in \cite{Gilkey}.
The $\eta$-invariant of $D$ is then defined by
$$
\eta(D) := \eta(D,0).
$$

We denote by $h(D)$ the (finite) dimension of the kernel of $D$ and define the $\xi$-invariant of $D$ by
$$
\xi(D) := \frac{\eta(D) + h(D)}{2},
$$
following Atiyah, Patodi, and Singer \cite{APS2}.
Note that $\eta(D)$ and $\xi(D)$ depend only on the spectrum of $D$.

\subsection{Characteristic forms}
We will use the following convention for characteristic forms: 
Given a local frame $(e_i)_{i \in \{1, \dotsc , N\}}$ for a vector bundle $E \to M$ with connection $\nabla$, the connection 1-form $\omega=(\omega_i^j)$ and curvature 2-form $\Omega=(\Omega_i^j)$ with respect to that frame are defined by
$$
\nabla_X e_i = \sum_{j=1}^N \omega^j_i(X) e_j
$$
and
$$
\nabla_X\nabla_Y e_i - \nabla_Y\nabla_X e_i - \nabla_{[X,Y]} e_i = \sum_{j=1}^N \Omega_i^j(X,Y) e_j.
$$
For a formal power series $q\in\R\llbracket x\rrbracket$, the form $q(\Omega/(2\pi i))$ is well defined because $\Omega$ is a matrix of 2-forms and hence nilpotent.
The associated additive characteristic form is given by 
$$
\tr\Big( q\Big(\frac{\Omega}{2\pi i}\Big) \Big) \in C^\infty(M,\Lambda^{\text{even}}T^*M).
$$
This form is closed and independent of the choice of local frame and hence globally defined.
For example, the Chern character form is the characteristic form associated with $q(x)=\exp(x)$.

Now let $\nabla^\bullet=(\nabla^s)_{s \in [a,b]}$ be a smooth 1-parameter family of connections on a bundle $E\to M$.
Fix a local frame (independent of $s$) and denote the resulting connection 1-form of $\nabla^s$ by $\omega^s$ and the curvature 2-form by $\Omega^s$.
Then the corresponding odd characteristic form is defined by
$$
\int_a^b \tr\Big(\frac{\partial_s \omega^s}{2\pi i} \wedge q'\Big(\frac{\Omega^s}{2\pi i}\Big)\Big) \, ds
\in C^\infty(M,\Lambda^{\text{odd}}T^*M) .
$$
For example, the odd Chern character form is given by
$$
\cs(\nabla^\bullet) 
= 
\int_a^b \tr\Big(\frac{\partial_s \omega^s}{2\pi i} \wedge \exp\Big(\frac{\Omega^s}{2\pi i}\Big)\Big) \, ds .
$$
Again, the odd characteristic form is independent of the choice of local frame and hence globally defined.
It is also invariant under reparametrizations of the family $\nabla^\bullet$.
More precisely, let $\phi: [c,d] \to [a,b]$ be a smooth function with $\phi(c)=a$ and $\phi(d)=b$.
Then we find, substituting $\sigma=\phi(s)$,
\begin{align}
\int_c^d \tr\Big(\frac{\partial_s \omega^{\phi(s)}}{2\pi i} \wedge q'\Big(\frac{\Omega^{\phi(s)}}{2\pi i}\Big)\Big) \, ds 
&=
\int_c^d \tr\Big(\frac{\partial_\sigma \omega^{\sigma}}{2\pi i}\frac{d\sigma}{ds} \wedge q'\Big(\frac{\Omega^{\sigma}}{2\pi i}\Big)\Big) \, ds \notag\\
&=
\int_a^b \tr\Big(\frac{\partial_\sigma \omega^{\sigma}}{2\pi i} \wedge q'\Big(\frac{\Omega^{\sigma}}{2\pi i}\Big)\Big) \, d\sigma .
\label{eq.parameterinvariant}
\end{align}
Thus the families $\nabla^\bullet$ and $\nabla^{\phi(\bullet)}$ have the same odd characteristic form.

If we pull back the bundle $E$ together with the family of connections $\nabla^\bullet$ via a smooth map $f: N \to M$, then the curvature and connection forms pull back accordingly when taken with respect to the pulled back local frame.
Therefore, the odd characteristic form is natural under pullbacks.
In particular, this means for the odd Chern character form that
$$
\cs(f^*\nabla^\bullet) = f^*\cs(\nabla^\bullet).
$$

\subsection{The Atiyah-Patodi-Singer index theorem}

Let $(X,g)$ be an even-dimensional compact Riemannian spin manifold with boundary $\partial X$ and $E \to X$ a hermitian vector bundle with metric connection $\nabla^E$. 
Denote the twisted Dirac operator of $X$ acting on sections of $\Sigma X \otimes E$ by $D^E$ and the twisted Dirac operator of $\partial X$ acting on sections of $\Sigma \partial X \otimes E$ by $A^E$.
Since $\nabla^E$ is metric, both operators are self-adjoint.

We say that $(X,\nabla^E)$ is of product type near $\partial X$ if there is a neighbourhood $Y$ of $\partial X$ in $X$, a diffeomorphism $\psi\colon \partial X\times [0,\varepsilon)\to Y$ and a vector bundle isomorphism $\Psi\colon \pi^*(E\vert_{\partial X})\to E\vert_Y$ covering $\psi$ where $\pi\colon \partial X\times [0,\varepsilon)\to \partial X$ is the projection such that
\begin{itemize}
\item 
$\psi^*g = \iota^*g + dt^2$ where $\iota\colon \partial X \to X$ is the inclusion and $t\in[0,\varepsilon)$ is the standard coordinate,
\item
$\Psi^*\nabla^E = \pi^*(\nabla^E\vert_{\partial X})$.
\end{itemize}

Since $X$ is even-dimensional, the spinor bundle $\Sigma X$ splits into the bundles of spinors of positive and negative chirality, $\Sigma X = \Sigma^+ X \oplus \Sigma^- X$.
With respect to the splitting
$$
\Sigma X \otimes E = (\Sigma^+ X \otimes E) \oplus (\Sigma^- X \otimes E),
$$
the twisted Dirac operator $D^E$ has the form
$$
D^E = \begin{pmatrix}0 & D^{E,-} \\ D^{E,+} & 0 \end{pmatrix} .
$$
We consider $D^{E,+}$ as a bounded operator from the Sobolev space $H^{1}(X; \Sigma^+ X \otimes E)$ to $L^2(X; \Sigma^- X \otimes E)$.

There is a canonical identification $\Sigma^+ X\vert_{\partial X} = \Sigma \partial X$ of Hermitian vector bundles.
Hence we have
$$
(\Sigma^+ X \otimes E)\vert_{\partial X} = \Sigma \partial X \otimes E\vert_{\partial X}.
$$
For any subset $I\subset\R$ denote its characteristic function by $\chi_I\colon \R\to\{0,1\}$.
Consider the spectral projector
$$
P := \chi_{[0,\infty)}(A^E) : L^2(\partial X; \Sigma \partial X \otimes E\vert_{\partial X}) \to L^2(\partial X; \Sigma \partial X \otimes E\vert_{\partial X}).
$$
We say that $u\in H^1(X; \Sigma^+ X \otimes E)$ satisfies the Atiyah-Patodi-Singer (APS) boundary condition if $P(u\vert_{\partial X})=0$.
We write $D^{E,+}_\APS$ for the operator $D^{E,+}$ with domain consisting of all $u\in H^1(X; \Sigma^+ X \otimes E)$ satisfying the APS boundary condition and codomain $L^2(X; \Sigma^- X \otimes E)$.

\begin{theorem}[Atiyah-Patodi-Singer index theorem \cite{APS1}*{Thm.~3.10}]
The operator $D^{E,+}_\APS$ is Fredholm.
If $(X,\nabla^E)$ is of product type near $\partial X$, then its index is given by
$$
\ind\big(D^{E,+}_\APS\big)  
= 
\int_X \Ahat(TX, g) \wedge \ch(\nabla^E)  - \xi(A^E).
$$
\end{theorem}
Here $\Ahat(TX, g)$ denotes the $\Ahat$-form of $TX$ equipped with the Levi-Civita connection for $g$.

\subsection{Spectral flow}
Let $D^\bullet=(D^s)_{s \in [a,b]}$ be a smooth family of self-adjoint elliptic differential operators of positive order acting on sections of a vector bundle over a closed manifold.
Smoothness of the family means that locally the coefficients of the differential operators depend smoothly on the parameter $s$.

Intuitively, the spectral flow of the family $D^\bullet$ counts how many negative eigenvalues become nonnegative minus the number of nonnegative eigenvalues becoming negative as the parameter $s$ runs from $a$ to $b$.
We describe the spectral flow following \cite{Phillips}.
If there exists an $a>0$ such that $\pm a$ are not in the spectrum of $D^s$ for any $s \in [a,b]$, then the rank of the spectral projector $\chi_{[-a,a]}(D^s)$ is finite and constant in $s$.
In this case, the spectral flow is given by 
$$
\sf(D^\bullet) = \rk\big(\chi_{[0,a]}(D^b)\big) - \rk\big(\chi_{[0,a]}(D^a)\big).
$$
In general, one can subdivide the interval $[a,b]$ as $a=s_0<s_1<\cdots<s_N=b$ and find $a_i>0$ such that $\pm a_i$ are not in the spectrum of $D^s$ for any $s \in [s_{i-1},s_{i}]$.
Then the spectral flow is given by
$$
\sf(D^\bullet) = \sum_{i=1}^N \Big( \rk\big(\chi_{[0,a_i]}(D^{s_i})\big) - \rk\big(\chi_{[0,a_i]}(D^{s_{i-1}})\big) \Big).
$$
This definition does not depend on the choices.
The spectral flow is invariant under homotopies of the family $D^\bullet$ with fixed endpoints.
A detailed discussion of the spectral flow can be found in the book \cite{SpectralFlowBook}.

\section{Proof of the theorem}
\label{sec.Proof} 

Having introduced the necessary concepts and tools, we can now proceed to the proof of the theorem.
Let $(M,g)$ be a closed odd-dimensional Riemannian spin manifold and $\nabla^\bullet=(\nabla^s)_{s \in [a,b]}$ a smooth 1-parameter family of metric connections on a Hermitian vector bundle $E \to M$. 
Put
$$
X := M \times [a,b]
$$
and denote by $\pi: X \to M$ the canonical projection. 
We equip $X$ with the product metric $\bar{g}=\pi^*g + dt^2$ where $t \in [a,b]$ is the standard coordinate.
We pull back the bundle $E$ and obtain the bundle $\bar{E}=\pi^*E \to X$.
Each section $\bar{u}$ of $\bar{E}$ can be identified with a family of sections $u$ of $E$ depending on the parameter $s\in[a,b]$.
We define a connection $\bar{\nabla}$ on $\bar{E}$ by
$$
\bar{\nabla}_{v + \alpha \partial_t\vert_s} \bar{u} = \nabla_v^s u + \alpha \partial_s u
$$
for $v\in TM$ and $\alpha\in\R$.
    
\begin{lemma}
Let $e_1,...,e_N$ be a local frame of $E$.
Let $\omega^s$ be the connection 1-form and $\Omega^s$ the curvature 2-form of $\nabla^s$ with respect to that frame.
Then $\bar{e}_1=e_1\circ\pi,...,\bar{e}_N=e_N\circ\pi$ is a local frame of $\bar{E}$ with respect to which the curvature 2-form of $\bar{\nabla}$ is given by
$$
\bar\Omega = \pi^*\Omega^{t} + dt \wedge \pi^*\partial_t\omega^{t}.
$$
\end{lemma}

\begin{proof}
For $v + \alpha \partial_t \in T_{(p,s)}X=T_pM\oplus\R\cdot\partial_t\vert_s$ we have
\begin{align*}
\bar{\nabla}_{v + \alpha \partial_t} \bar e_i 
= 
\nabla_v^s e_i  + \alpha \partial_t|_{t=s} e_i
= 
\nabla_v^s e_i 
= 
\sum_{j=1}^N (\omega^s)_i^j(v) \bar e_j.
\end{align*}
Hence,
\begin{align*}
\bar\omega(v + \alpha \partial_t\vert_s) = \omega^s(v).
\end{align*}
Let $v + \alpha\partial_t, w + \beta\partial_t \in T_{(p,s)}X$.
We extend both tangent vectors to local vector fields which are synchronous at the point $(p,s)$.
Using the structural equation $\bar{\Omega}=d\bar{\omega} + \bar{\omega} \wedge \bar{\omega}$, we compute:
\begin{align*}
\bar\Omega^i_k&(v + \alpha \partial_t, w + \beta \partial_t) 
=
d\bar\omega^i_k(v + \alpha \partial_t, w + \beta\partial_t) + (\bar\omega \wedge \bar \omega)^i_k (v + \alpha\partial_t, w + \beta\partial_t) \\
&=
(v + \alpha \partial_t) (\bar\omega^i_k(w + \beta\partial_t)) - (w + \beta\partial_t)(\bar\omega^i_k(v + \alpha\partial_t)) + \sum_{j=1}^N (\bar\omega^i_j \wedge \bar\omega^j_k) (v + \alpha\partial_t, w + \beta\partial_t) \\
&=
v((\omega^{s})^i_k(w)) - w((\omega^{s})^i_k(v)) + \alpha \partial_t|_{t=s} (\omega^{t})^i_k(w) - \beta \partial_t|_{t=s} (\omega^{t})^i_k(v)
 + \sum_{j=1}^N ((\omega^{s})^i_j \wedge (\omega^{s})^j_k)(v,w) \\
&=
d(\omega^{s})^i_k(v,w) + \alpha \partial_t|_{t=s} (\omega^{t})^i_k(w) - \beta \partial_t|_{t=s}(\omega^{t})^i_k(v) + (\omega^{s} \wedge \omega^{s})^i_k(v,w) \\
&=
(\Omega^{s})^i_k(v,w) + \alpha \partial_t|_{t=s} (\omega^{t})^i_k(w) - \beta \partial_t|_{t=s}(\omega^{t})^i_k(v) .
\end{align*}
Employing
\begin{align*}
(dt \wedge \partial_t \omega^{t})(v + \alpha\partial_t, w + \beta\partial_t) 
&= 
\det \begin{pmatrix} dt(v + \alpha \partial_t) & (\partial_t \omega^{t})(v + \alpha\partial_t) \\ dt(w + \beta\partial_t) & (\partial_t\omega^{t})(w + \beta\partial_t) \end{pmatrix} \\
&= 
\det\begin{pmatrix} \alpha & \partial_t \omega^{t} (v) \\ \beta & \partial_t \omega^{t}(w) \end{pmatrix} \\
&= 
\alpha \partial_t \omega^{t} (w) - \beta \partial_t\omega^{t}(v)
\end{align*}
we find
\begin{equation*}
\bar\Omega = \pi^*\Omega^{t} + dt \wedge \pi^*\partial_t\omega^{t}.
\qedhere
\end{equation*}
\end{proof}

\begin{lemma}\label{lem.OmegaOben}
For any power series $q\in\R\llbracket x\rrbracket$, we have
$$
\tr\Big( q\Big(\frac{\bar\Omega}{2\pi i}\Big) \Big)
=
\pi^*\tr\Big( q\Big(\frac{\Omega^{t}}{2\pi i}\Big) \Big) + dt \wedge \pi^*\tr\Big(\frac{\partial_t \omega^{t}}{2\pi i} \wedge q'\Big(\frac{\Omega^{t}}{2\pi i}\Big)\Big).
$$
\end{lemma}

\begin{proof}
For each $n\in\N$ we have
\begin{align*}
\bar{\Omega}^n
&=
\big(\pi^*\Omega^{t} + dt \wedge \pi^*\partial_t\omega^{t}\big)^n 
=
\big(\pi^*\Omega^{t}\big)^n + \sum_{k=1}^n\big(\pi^*\Omega^{t}\big)^{k-1} \wedge \big(dt \wedge \pi^*\partial_t\omega^{t}\big) \wedge \big(\pi^*\Omega^{t}\big)^{n-k} 
\end{align*}
and therefore
\begin{align*}
\tr(\bar{\Omega}^n)
&=
\tr\big(\big(\pi^*\Omega^{t}\big)^n\big) + n \tr\big(\big(dt \wedge \pi^*\partial_t\omega^{t}\big) \wedge \big(\pi^*\Omega^{t}\big)^{n-1} \big) \\
&=
\pi^*\tr\big((\Omega^{t})^n\big) + dt \wedge \pi^*\tr\big(\partial_t\omega^{t} \wedge n(\Omega^{t})^{n-1} \big) .
\end{align*}
The claim follows.
\end{proof}

\begin{example}
For $q=\exp$ we have $q'=\exp$ and hence
\begin{equation}
\ch(\bar{\nabla}) = \pi^*\ch(\nabla^{t}) + dt \wedge \pi^*\tr\Big(\frac{\partial_t \omega^{t}}{2\pi i} \wedge \exp\Big(\frac{\Omega^{t}}{2\pi i}\Big)\Big).
\label{eq.chOben}
\end{equation}
\end{example}

\begin{remark}
It is well known that the characteristic forms $\tr\Big( q\Big(\frac{\bar\Omega}{2\pi i}\Big) \Big)$ are closed.
Therefore, Lemma~\ref{lem.OmegaOben} implies
\begin{align*}
0
&=
d\tr\Big( q\Big(\frac{\bar\Omega}{2\pi i}\Big) \Big) \\
&=
d\Big[\pi^*\tr\Big( q\Big(\frac{\Omega^{t}}{2\pi i}\Big) \Big) + dt \wedge \pi^*\tr\Big(\frac{\partial_t \omega^{t}}{2\pi i} \wedge q'\Big(\frac{\Omega^{t}}{2\pi i}\Big)\Big)\Big]\\
&=
\pi^*d\tr\Big( q\Big(\frac{\Omega^{t}}{2\pi i}\Big) \Big) - dt \wedge \pi^*d\tr\Big(\frac{\partial_t \omega^{t}}{2\pi i} \wedge q'\Big(\frac{\Omega^{t}}{2\pi i}\Big)\Big)\\
&=
- \pi^*d\tr\Big(\frac{\partial_t \omega^{t}}{2\pi i} \wedge q'\Big(\frac{\Omega^{t}}{2\pi i}\Big)\Big) \wedge dt.
\end{align*}
Thus,
$$
0=\pi^*d\int_a^b \tr\Big(\frac{\partial_t \omega^{t}}{2\pi i} \wedge q'\Big(\frac{\Omega^{t}}{2\pi i}\Big)\Big) \, dt,
$$
which implies
$$
d\int_a^b \tr\Big(\frac{\partial_t \omega^{t}}{2\pi i} \wedge q'\Big(\frac{\Omega^{t}}{2\pi i}\Big)\Big) \, dt = 0,
$$
i.e., the odd characteristic form is closed as well.
\end{remark}

We denote the Dirac operator on $X$ twisted with $(\bar{E},\bar{\nabla})$  by $\bar{D}^{\bar{E}}$ and the one on $M$ twisted with $(E,\nabla^t)$ by~$D^t$.

\begin{lemma}\label{lem.sf=ind}
We have
$$
\sf(D^\bullet) = \ind\big(\bar{D}^{\bar{E},+}_\APS\big) + h(D^b).
$$
\end{lemma}

\begin{proof}
The operator $\bar{D}^{\bar{E},+}$ can be written as 
$$
\bar{D}^{\bar{E},+} =  \gamma(\partial_t + D^t)
$$
where $\gamma$ denotes the principal symbol of $\bar{D}^{\bar{E},+}$ evaluated on the covector $dt$.
In particular, $\gamma$ is an isomorphism.
Thus, $\bar{D}^{\bar{E},+}$ has the same index as the operator $\D:=\partial_t + D^t$.
The APS boundary conditions now read as $\chi_{[0,\infty)}(D^a)(u\vert_{t=a}) =0$ and $\chi_{(-\infty,0]}(D^b)(u\vert_{t=b}) =\chi_{[0,\infty)}(-D^b)(u\vert_{t=b}) =0$.
The opposite signs for $t=a$ and $t=b$ are a consequence of $\partial_t$ being the interior unit normal field along the boundary component $\{t=a\}$ of $X$ and the ourward unit normal field along the component $\{t=b\}$.

Considering the modified APS boundary condition (mAPS) 
$$
\chi_{[0,\infty)}(D^a)(u\vert_{t=a}) =0\quad\text{ and }\quad\chi_{(-\infty,0)}(D^b)(u\vert_{t=b}) =0,
$$
it follows from \cite{Ronge}*{Theorem~4.9} that
$$
\sf(D^\bullet)=\ind(\D_\mAPS).
$$
Corollary~8.8 in \cite{BB} implies
$$
\ind(\D_\mAPS) = \ind(\D_\APS) + h(D^b).
$$
This combines to yield
\begin{equation*}
\sf(D^\bullet)=\ind(\D_\APS) + h(D^b)=\ind\big(\bar{D}^{\bar{E},+}_\APS\big) + h(D^b).
\qedhere
\end{equation*}
\end{proof}

\begin{proof}[\textbf{\emph{Concluding the proof of the theorem}}]
Since $\bar{g}$ is the product metric on $X$, we have $\Ahat(TX,\bar{g})=\pi^*\Ahat(TM,g)$.
We choose a smooth function $\phi: [a,b] \to [a,b]$ with $\phi=a$ near $a$ and $\phi=b$ near $b$.
We use the family $\nabla^{\phi(\bullet)}$ to induce a connection $\bar{\nabla}$ on $\bar{E}=\pi^*E$.
Then $(X,\bar{\nabla})$ is of product type near $\partial X$ and we can apply the Atiyah-Patodi-Singer index theorem.
Lemma~\ref{lem.sf=ind}, the homotopy invariance of the spectral flow, and the Atiyah-Patodi-Singer index theorem applied to $\bar{D}^{\bar{E},+}_\APS$ yield
\begin{align*}
\sf(D^\bullet)
&=
\sf(D^{\phi(\bullet)})
=
\ind\big(\bar{D}^{\bar{E},+}_\APS\big) + h(D^b)
=
\int_X \Ahat(TX, \bar{g}) \wedge \ch(\bar{\nabla}) - (\xi(D^a)+\xi(-D^b)) + h(D^b) .
\end{align*}
For the integral we find, using \eqref{eq.parameterinvariant} and \eqref{eq.chOben},
\begin{align*}
\int_X \Ahat(TX, \bar{g}) \wedge \ch(\bar{\nabla})
&=
\int_X \pi^*\Ahat(TM, g) \wedge \Big[\pi^*\ch(\nabla^{\phi(t)}) + dt \wedge \pi^*\tr\Big(\frac{\partial_t \omega^{\phi(t)}}{2\pi i} \wedge \exp\Big(\frac{\Omega^{\phi(t)}}{2\pi i}\Big)\Big)\Big] \\
&=
\int_X \pi^*\Ahat(TM, g) \wedge  dt \wedge \pi^*\tr\Big(\frac{\partial_t \omega^{\phi(t)}}{2\pi i} \wedge \exp\Big(\frac{\Omega^{\phi(t)}}{2\pi i}\Big)\Big) \\
&=
-\int_X \pi^*\Ahat(TM, g) \wedge \pi^*\tr\Big(\frac{\partial_t \omega^{\phi(t)}}{2\pi i} \wedge \exp\Big(\frac{\Omega^{\phi(t)}}{2\pi i}\Big)\Big) \wedge  dt \\
&=
-\int_M \Ahat(TM, g) \wedge \int_a^b\tr\Big(\frac{\partial_t \omega^{\phi(t)}}{2\pi i} \wedge \exp\Big(\frac{\Omega^{\phi(t)}}{2\pi i}\Big)\Big) dt \\
&=
-\int_M \Ahat(TM, g) \wedge \cs(\nabla^{\phi(\bullet)}) \\
&=
-\int_M \Ahat(TM, g) \wedge \cs(\nabla^\bullet) .
\end{align*}
The boundary term becomes
\begin{align*}
- (\xi(D^a)+\xi(-D^b)) + h(D^b) 
&=
- \frac{\eta(D^a)+h(D^a)+\eta(-D^b)+h(-D^b)}{2} + h(D^b) \\
&=
- \frac{\eta(D^a)+h(D^a)-\eta(D^b)-h(D^b)}{2} \\
&=
\xi(D^b)-\xi(D^a) .
\end{align*}
This concludes the proof of the theorem.
\end{proof}

\section{Examples and an application}

The first example has been discussed by Getzler, using somewhat different wording, compare Theorem~2.8 in \cite{G}.

\begin{numberedexample}\label{ex.UN}
Let $\omega$ be the Maurer-Cartan form on the unitary group $\UN$.
This is a $1$-form on $\UN$ with values in the Lie algebra $\mathfrak{u}(N)$ of skew-Hermitian $N\times N$-matrices.
Consider the family of connections $\nabla^s = d + s\omega$ on the trivial bundle $E = \UN \times \C^N \to \UN$.
Then $s\omega$ is the connection $1$-form of $\nabla^s$ with respect to the global constant standard frame.
The structural equation $d\omega + \omega \wedge \omega = 0$ of the Maurer-Cartan form implies for the curvature 1-form of $\nabla^s$ that 
\begin{align*}
\Omega^s
&=
d(s\omega) + (s\omega) \wedge (s\omega) 
=
s \, d\omega + s^2 \, \omega \wedge \omega  
=
s(s-1) \, \omega \wedge \omega .
\end{align*}
We compute the odd Chern character form of the family $\nabla^\bullet=(\nabla^s)_{s\in[0,1]}$:
\begin{align*}
\cs(\nabla^\bullet)
&=
\int_0^1 \tr\Big(\frac{\partial_s (s\omega)}{2\pi i} \wedge \exp\Big(\frac{\Omega^s}{2\pi i}\Big)\Big) \, ds \\
&=
\int_0^1 \tr\Big(\frac{\omega}{2\pi i} \wedge \exp\Big(\frac{s(s-1) \, \omega^2}{2\pi i}\Big)\Big) \, ds\\
&=
\int_0^1 \tr\Big(\frac{\omega}{2\pi i} \wedge \sum_{k=0}^\infty\frac{1}{k!}\Big(\frac{s(s-1) \, \omega^2}{2\pi i}\Big)^k\Big) \, ds\\
&=
\sum_{k=0}^\infty\frac{1}{k!}\frac{1}{(2\pi i)^{k+1}}\tr\big(\omega^{2k+1}\big) \int_0^1s^k(s-1)^k \, ds\\
&=
\sum_{k=0}^\infty\frac{1}{k!}\frac{1}{(2\pi i)^{k+1}}\tr\big(\omega^{2k+1}\big) \frac{(-1)^k (k!)^2}{(2k+1)!}\\
&=
\sum_{k=0}^\infty\frac{ (-1)^k k!}{(2k+1)!}\frac{\tr\big(\omega^{2k+1}\big)}{(2\pi i)^{k+1}}  .
\end{align*}
By \cite{CS}*{\S 3}, the cohomology classes 
$$
\zeta_{2k+1} := \bigg[\frac{ (-1)^k k!}{(2k+1)!}\frac{\tr\big(\omega^{2k+1}\big)}{(2\pi i)^{k+1}}\bigg]\in H_\deR^{2k+1}(\UN)
$$
are nontrivial for $k=0,1,2,...,N-1$ and generate the de Rham cohomology of $\UN$ as an algebra,
$$
H_\deR^\bullet(\UN) = \Lambda[\zeta_1,\zeta_3,...,\zeta_{2N-1}].
$$
By Bott periodicity (\cite{Bott}*{Equation~(1.5)}), the odd homotopy groups of $\UN$ satisfy $\pi_{2k+1}(\UN) \cong \Z$ for $k=0,1,2,...,N-1$.
We choose a generator $g_{2k+1}$ of $\pi_{2k+1}(\UN)$.
The Hurewicz homomorphism 
$$
h_{2k+1}\colon \pi_{2k+1}(\UN) \to H_{2k+1}(\UN;\Z)
$$
maps it to a class pairing nontrivially with $\zeta_{2k+1}$, see e.g.\ the appendix in \cite{MiMo}.
Put $\alpha_k:=\langle \zeta_{2k+1},h_{2k+1}(g_{2k+1})\rangle$.

Now, let $0\le k\le N-1$ and let $f\colon S^{2k+1} \to \UN$ be a smooth map.  
We pull back the trivial $\C^N$-bundle with the family of connections $\nabla^s$ via $f$ and consider the corresponding family of twisted Dirac operators $D^{f,s}$ on $S^{2k+1}$ where $s\in [0,1]$.
For $s=0$ and $s=1$ the connection $\nabla^s$ is flat so that the resulting Dirac operators $D^{f,0}$ and $D^{f,1}$ have the same spectrum as the untwisted Dirac operator on $S^{2k+1}$ with all multiplicities multiplied by $N$.
Therefore the $\xi$-invariants of $D^{f,0}$ and $D^{f,1}$ cancel.

Let $f$ represent $m\cdot g_{2k+1}\in\pi_{2k+1}(\UN)$ where $m\in\Z$.
Using that the $\Ahat$-class of the sphere is trivial, we compute
\begin{align*}\sf(D^{f,\bullet})
&=
-\int_{S^{2k+1}} \Ahat(TS^{2k+1}, g) \wedge \cs(f^*\nabla^\bullet) 
=
-\int_{S^{2k+1}} f^*\cs(\nabla^\bullet) \\
&=
-\langle [S^{2k+1}], f^*\zeta_{2k+1} \rangle 
=
-\langle f_*[S^{2k+1}], \zeta_{2k+1} \rangle 
=
-\langle m \, h_{2k+1}(g_{2k+1}), \zeta_{2k+1} \rangle
= 
-m\alpha_k .
\end{align*}
Choosing $m=1$, we see in particular that $\alpha_k \in \Z$.
\end{numberedexample}

\begin{numberedexample}\label{ex.Sn-neu}
Let $S\subset\R^{n+1}$ be a closed connected hypersurface. 
Let $\nu$ be a unit normal field along $S$.
The Weingarten map $W\colon TS \to TS$ is given by $W(X) = -\nbR_X \nu$ where $\nbR$ is the Levi-Civita connection of $\R^{n+1}$.
For example, if $S=S^n$ is the standard sphere and $\nu$ is the inward unit normal field then $W=\id$.

If $n$ is even, then the restriction of spinor bundle on $\R^{n+1}$ to $S$ can be naturally identified with the spinor bundle of $S$, i.e., $\Sigma\R^{n+1}|_S = \Sigma S$.
If $n$ is odd, then we have $\Sigma^+\R^{n+1}|_S = \Sigma S$.
In both cases, Clifford multiplication on $S$ is given by 
$$
\gamma(X) = c_0(X)c_0(\nu)\quad\text{ for }\quad X\in TS.
$$
Here $c_0$ denotes Clifford multiplication on $\R^{n+1}$.
The spinor connections are related by
$$
\nbR_X\phi = \nbS_X \phi + \tfrac12 c_0(W(X))c_0(\nu)\phi = \nbS_X \phi + \tfrac12 \gamma(W(X))\phi,
$$
see e.g. \cite{Baer3}*{Prop.~2.1}.
Hence if $\phi$ is a parallel spinor field on $\R^{n+1}$, then its restriction to $S$ satisfies
$$
\nbS_X \phi = -\tfrac12 \gamma(W(X))\phi.
$$
Such spinors are called generalized Killing spinors on $S$ w.r.t.\ the endomorphism field $-W$.
If $\phi=c_0(\nu)\tilde{\phi}$ for a parallel spinor field $\tilde{\phi}$ on $\R^{n+1}$, then we find
$$
\nbR_X \phi = \nbR_X(c_0(\nu)\tilde{\phi}) = c_0(\nbR_X \nu)\tilde{\phi} = -c_0(W(X))\tilde{\phi} = c_0(W(X))c_0(\nu)^2\tilde{\phi} = \gamma(W(X))\phi
$$
and 
$$
\nbR_X \phi = \nbS_X \phi + \tfrac12 \gamma(W(X))\phi,
$$
hence
$$
\nbS_X \phi = \tfrac12 \gamma(W(X))\phi.
$$
Thus, in this case $\phi$ is a generalized Killing spinor on $S$ w.r.t.\ the endomorphism field $+W$.
This shows that the spinor bundle of $S$ can be trivialized by generalized Killing spinors w.r.t.\ the endomorphism field $W$ and also by generalized Killing spinors w.r.t.\ the endomorphism field $-W$.

We introduce the $1$-parameter family of connections $\nabla^s = \nabla + s\gamma\circ W$ on $\Sigma S$ where $s\in[-\frac12,\frac12]$ and $\nabla$ is the connection induced by the Levi-Civita connection on $TS$.
Then $\nabla^{\pm\nicefrac12}$ are flat connections.

We choose an orthonormal frame of the trivial bundle $\Sigma S$ and denote by $\omega$ the connection $1$-form of $\nabla$ with respect to that frame.
Then $\omega^s = \omega + s\gamma\circ W$ is the connection $1$-form of $\nabla^s$ where we do not distinguish between $\gamma$ as an endomorphism-valued $1$-form and as a matrix-valued $1$-form with respect to the chosen frame.

Denote by $\Omega^s$ the curvature $2$-form of $\nabla^s$.
We then have
\begin{align}
\Omega^s
&=
d\omega^s + \omega^s \wedge \omega^s
=
d\omega + \omega \wedge \omega + s \, d(\gamma\circ W) + s \, (\omega \wedge (\gamma\circ W) + (\gamma\circ W) \wedge \omega) + s^2 \, (\gamma\circ W) \wedge (\gamma\circ W) .
\label{eq.Omegas}
\end{align}
The connections $\nabla^{\pm\nicefrac12}$ being flat gives us
\begin{align}
0 = \Omega^{\nicefrac12}
&=
d\omega + \omega \wedge \omega + \tfrac12 \, d(\gamma\circ W) + \tfrac12 \, (\omega \wedge (\gamma\circ W) + (\gamma\circ W) \wedge \omega) + \tfrac14 \, (\gamma\circ W) \wedge (\gamma\circ W) ,
\label{eq.Omega12}\\
0 = \Omega^{-\nicefrac12}
&=
d\omega + \omega \wedge \omega - \tfrac12 \, d(\gamma\circ W) - \tfrac12 \, (\omega \wedge (\gamma\circ W) + (\gamma\circ W) \wedge \omega) + \tfrac14 \, (\gamma\circ W) \wedge (\gamma\circ W) .
\label{eq.Omega-12}
\end{align}
Adding \eqref{eq.Omega12} and \eqref{eq.Omega-12} yields $d\omega + \omega \wedge \omega = -\frac14(\gamma\circ W)\wedge(\gamma\circ W)$ and subtracting them gives $d(\gamma\circ W) = -(\omega \wedge (\gamma\circ W) + (\gamma\circ W) \wedge \omega )$.
Inserting this back into \eqref{eq.Omegas} we find
\begin{align}
\Omega^s
&=
-\tfrac14 (\gamma\circ W) \wedge (\gamma\circ W)  + s^2 \, (\gamma\circ W) \wedge (\gamma\circ W) 
=
(s - \tfrac12)(s + \tfrac12) \, (\gamma\circ W) \wedge (\gamma\circ W) . 
\label{eq.OmegaFinal}
\end{align}
The same computation as in Example~\ref{ex.UN} now gives for the family $\nabla^\bullet=(\nabla^s)_{s\in[-\nicefrac12,\nicefrac12]}$:
\begin{align*}
\cs(\nabla^\bullet)
&=
\sum_{k=0}^\infty\frac{ (-1)^k k!}{(2k+1)!}\frac{\tr\big((\gamma\circ W)^{2k+1}\big)}{(2\pi i)^{k+1}}  .
\end{align*}

From now on, we assume that $S$ is diffeomorphic to $S^n$.
Then the cohomology of $S$ is trivial in all degrees except $0$ and $n$, only the term with $2k+1=n$ can give a nonexact contribution to the odd Chern character.
In particular, this happens only if the dimension $n$ of $S$ is odd.

Let $n=2m+1$ be odd, let $M$ be an $n$-dimensional closed Riemannian spin manifold and let $f\colon M\to S$ be a smooth map.
We pull back the bundle $\Sigma S$ along $f$ together with the family of connections $\nabla^\bullet$  and consider the corresponding family of twisted Dirac operators $D^{f,s}$ on $M$ where $s\in [-\nicefrac12,\nicefrac12]$.
Then 
$$
\Ahat(TM)\wedge\cs(f^*\nabla^\bullet)
=
\Ahat(TM)\wedge f^*\cs(\nabla^\bullet)
=
\frac{ (-1)^m m!}{(2m+1)!}\frac{f^*\tr\big((\gamma\circ W)^{2m+1}\big)}{(2\pi i)^{m+1}}
$$
because all nontrivial components of $\Ahat(TM)$ lead to too high degree forms.

Let $e_1,...,e_n$ be a positively oriented orthonormal tangent frame of $S$ consisting of eigenvectors of $W$.
The eigenvalues $\kappa_1,...,\kappa_n$ of $W$ are the principal curvatures of $S$.
Then the endomorphisms $(\gamma\circ W)(e_1)=\kappa_1 \gamma(e_1),...,(\gamma\circ W)(e_n)=\kappa_n \gamma(e_n)$ anticommute.
We observe
\begin{align*}
(\gamma\circ W)^{2m+1}(e_1,...,e_{2m+1})
&=
\sum_{\sigma\in S_{2m+1}} \sign(\sigma)(\gamma\circ W)(e_{\sigma(1)}) \circ \cdots \circ (\gamma\circ W)(e_{\sigma(2m+1)}) \\
&=
(2m+1)! \cdot (\gamma\circ W)(e_1) \circ \cdots \circ (\gamma\circ W)(e_{2m+1}) \\
&=
(2m+1)! \cdot \kappa_1\cdots\kappa_{2m+1} \cdot \gamma(e_1) \circ \cdots \circ \gamma  (e_{2m+1}) \\
&=
(2m+1)! \cdot \det(W)\cdot (-i)^{m+1} \cdot \id_{\Sigma S^n} .
\end{align*}
The last equality follows from the fact that the so-called complex volume element $i^{m+1} \gamma(e_1) \circ \cdots \circ \gamma(e_{2m+1})$ acts trivially on the spinor space in odd dimensions, see \cite{LM}*{Ch.~I, \S~5}.
Thus,
$$
\tr((\gamma\circ W)^{2m+1}) 
= 
(2m+1)! \cdot \det(W)\cdot (-i)^{m+1} \cdot \rk(\Sigma S^n) \cdot\vol_{S^n} 
= 
(2m+1)! \cdot \det(W)\cdot (-i)^{m+1} \cdot 2^{m}\cdot\vol_{S}
$$
where $\vol_{S}$ denotes the volume form of $S$.
Denote the volume form of the unit sphere $S^n$ by $\vol_{S^n}$ and regard $-\nu$ as a map from $S$ to $S^n$.
Then $\det(W)\vol_{S} = (-\nu)^*\vol_{S^n}$.

Since $D^{f,-\nicefrac12}$ and $D^{f,\nicefrac12}$ are both obtained by twisting the Dirac operator of $M$ with a flat trivial bundle, they are isospectral so that their $\xi$-invariants cancel.
Using the main theorem, we compute the spectral flow of the family $D^{f,\bullet}=(D^{f,s})_{s\in[-\nicefrac12,\nicefrac12]}$:
\begin{align*}
\sf(D^{f,\bullet})
&=
-\int_M \Ahat(TM) \wedge \cs(f^*\nabla^\bullet) \\
&=
-\frac{ (-1)^m m!}{(2m+1)!(2\pi i)^{m+1}} \int_M f^*\tr\big((\gamma\circ W)^{2m+1}\big) \\
&=
-\frac{ (-1)^m m!}{(2m+1)!(2\pi i)^{m+1}} (2m+1)! \cdot (-i)^{m+1} \cdot 2^{m}\cdot\int_M f^*(-\nu)^*\vol_{S^n} \\
&=
\frac{m!}{2\pi^{m+1}} \cdot\int_M f^*(-\nu)^*\vol_{S^n} \\
&=
\frac{m!}{2\pi^{m+1}} \cdot \deg(f)\cdot\deg(-\nu)\cdot\vol(S^n) \\
&=
\deg(f).
\end{align*}
\end{numberedexample}

In order to apply the results of Example~\ref{ex.Sn-neu} we need an auxiliary lemma.
For a linear map $F$ between Euclidean vector spaces we denote by $|F|_{\tr}$ its trace norm, i.e.\ the sum of its singular values, and by $|F|_\op$ its operator norm, i.e., the largest singular value.

\begin{lemma}[\cite{Baer}*{Lemma~4}]
\label{lem.smart}
Let $U$ and $V$ be $n$-dimensional Euclidean vector spaces, let $F\colon U\to V$ be linear and let $B\colon V\to V$ be symmetric and positive semidefinite.
Then
\begin{equation}
|B\circ F|_{\tr} \le \tr(B) |F|_\op.
\label{eq.normcomposition}
\end{equation}
If $B$ is positive definite and equality holds in \eqref{eq.normcomposition}, then $F$ is $|F|_\op$ times an isometry.
\end{lemma}

The following theorem was first proved by Llarull \cite{Ll}*{Theorem~C} for $S=S^n$.
In even dimensions it follows from work by Goette and Semmelmann (\cite{GS}*{Theorem~2.4}) under an additional pinching assumption on the principal curvatures of $S$.
In odd dimensions, it has first been shown by Li, Su, and Wang in \cite{LSW}*{Theorem~1.2} using a spectral flow computation.

\begin{theorem}
Let $S\subset\R^{n+1}$ be a closed connected hypersurface, where $n \geq 3$.
Assume that the Weingarten map w.r.t.\ the inward unit normal is positive definite at each point of $S$.
Let $(M,g)$ be an $n$-dimensional connected closed Riemannian spin manifold.
Let $f\colon M \to S$ be a smooth map which is $\Lambda^2$-$1$-contracting and satisfies $\deg(f) \neq 0$. 
Assume that the scalar curvatures of $M$ and $S$ satisfy $\scal_M \geq \scal_S\circ f$.

Then $f$ is a Riemannian isometry.
\end{theorem}

Here $\Lambda^2$-$1$-contracting means that the induced map $\Lambda^2 df(x)\colon \Lambda^2 T_xM \to \Lambda^2 T_{f(x)}S$ satisfies $|\Lambda^2 df|_\op \leq 1$ for each $x\in M$.

\begin{proof}
Let $E := f^*\Sigma S$, $\nabla^{\bullet}$ and $D^{f, \bullet}$ be as in Example~\ref{ex.Sn-neu}. 
The Bochner-Lichnerowicz-Weitzenböck formula reads
\begin{equation}
(D^{f,s})^2 = (\tilde\nabla^s)^*\tilde\nabla^s + \frac{\scal_M}{4} + \mathcal{R}^{E,s}
\label{eq.BLW}
\end{equation}
where $\tilde\nabla^s = \nabla^{\Sigma M} \otimes 1 + 1 \otimes f^*\nabla^s$ and the curvature term coming from the twist bundle is given in terms of an orthonormal tangent basis $e_1,...,e_n$ by
$$
\mathcal{R}^{E,s} = \frac12 \sum_{i \neq j} c(e_i)c(e_j) \otimes f^*\Omega^s(e_i,e_j).
$$
Here $c$ denotes the Clifford multiplication on $M$. 
We may choose the orthonormal basis of $T_xM$ in such a way that $W(df(e_i)) = \mu_i b_i$, where $b_1,...,b_n$ is an  orthonormal basis for $T_{f(x)}S^n$ and $0 \leq \mu_1 \leq \dotsc \leq \mu_n$ are the singular values of $W\circ df$ at $x$. 
Using \eqref{eq.OmegaFinal} we have for $i\neq j$ that
\begin{align*}
f^*\Omega^s(e_i,e_j) &
=
\Omega^s(df(e_i), df(e_j)) \\
& =
(s-\tfrac12)(s + \tfrac12)\big [\gamma (W (df(e_i)))\gamma (W (df(e_j))) - \gamma (W (df(e_j))) \gamma (W(df(e_i)))\big] \\
& =
(s-\tfrac12)(s + \tfrac12) \big[\mu_i\mu_j \gamma(b_i)\gamma(b_j) - \mu_i\mu_j\gamma(b_j)\gamma(b_i)\big] \\
& =
2(s-\tfrac12)(s + \tfrac12) \mu_i\mu_j \gamma(b_i)\gamma(b_j) .
\end{align*}
We estimate the operator norm of the endomorphism $\mathcal{R}^{E,s}$ at each point of $M$
\begin{align*}
|\mathcal{R}^{E,s}|_\op
& \leq 
|(s-\tfrac12)(s+\tfrac12)| \sum_{i \neq j} |c(e_i)c(e_j) \otimes \mu_i\mu_j\gamma(b_i)\gamma(b_j)|_\op \\
& \leq 
\tfrac14 \sum_{i \neq j} \mu_i\mu_j |c(e_i)c(e_j)|_\op |\gamma(b_i)\gamma(b_j)|_\op 
=
\tfrac14 \sum_{i \neq j} \mu_i\mu_j .
\end{align*}
Since $\Lambda^2(W\circ df)$ has the singular values $\mu_i\mu_j$ for $i < j$ we have $|\Lambda^2(W\circ df)|_{\tr} = \tfrac12 \sum_{i \neq j} \mu_i\mu_j$ and hence
\begin{equation}
|\mathcal{R}^{E,s}|_\op
\le
\tfrac12 |\Lambda^2(W\circ df)|_{\tr}
\label{in.R}
\end{equation}
By Lemma~\ref{lem.smart} we have at $x\in M$
\begin{equation}
\big|\Lambda^2(W\circ df)|_x\big|_{\tr} 
=
\big|\Lambda^2(W)|_{f(x)}\circ \Lambda^2(df)|_x\big|_{\tr} 
\le
\tr(\Lambda^2(W))|_{f(x)}\cdot \big|\Lambda^2(df)|_x\big|_{\op}
\le
\tr(\Lambda^2(W))|_{f(x)}
\label{in.Lambda2}
\end{equation}
because $f$ is $\Lambda^2$-$1$-contracting. 
Combining \eqref{in.R} and \eqref{in.Lambda2} we obtain
\begin{equation}
|\mathcal{R}^{E,s}|_\op
\le
\tfrac12 \tr(\Lambda^2(W)) \circ f
=
\tfrac14 \scal_S \circ f
\end{equation}
where the last equality follows from the Gauss equation.
Together with the assumption on the scalar curvatures this shows that $\mathcal{R}^{E,s}+\frac{\scal_M}{4}$ is a nonnegative endomorphism for each $s \in [-\tfrac12,\tfrac12]$.

If the dimension $n$ is odd, then $\sf(D^{f,\bullet})=\deg(f) \neq 0$ shows that  $D^{f,s}$ has nontrivial kernel for some $s \in [-\tfrac12,\tfrac12]$. 
If $n$ is even, a standard index computation shows that $D^{f,0}$ has nontrivial kernel.
In either case, let $\phi \in \ker(D^{f,s}) \setminus \{0\}$. 
Then \eqref{eq.BLW} yields
\begin{align}
0 
=
\int_M |D^{f,s} \phi|^2
=
\int_M \Big(|\tilde\nabla^s \phi|^2 + \frac{\scal_M}4 |\phi|^2 + \langle \mathcal{R}^{E,s} \phi, \phi\rangle\Big) 
\geq 
\int_M |\tilde\nabla^s \phi|^2 \geq 0.
\label{eq.BLWintegrated}
\end{align}
We infer that $\tilde\nabla^s \phi = 0$ and hence $|\phi|$ is constant (and nonzero). 
Furthermore, all above inequalities need to be equalities.
In particular, we have equality in \eqref{in.Lambda2}.
Lemma~\ref{lem.smart} implies that $\Lambda^2 df$ is a multiple of a linear isometry at each point.

Since we have equality in $|\Lambda^2(df)|_{\op} \leq 1$ it follows that $\Lambda^2 df$ is actually a linear isometry.
Thus, at each point $x \in M$ the singular values $\lambda_i$ of $df(x)$ satisfy $\lambda_i \lambda_j = 1$ for $i \neq j$.
Since $n \geq 3$ this implies $\lambda_1 = \cdots = \lambda_n = 1$.
Therefore, $df(x)$ is itself a linear isometry.
In other words, $f$ is a local isometry.
Surjectivity yields that $f$ is a Riemannian covering and $1$-connectedness of $S$ that $f$ is a Riemannian isometry.
\end{proof}

\begin{remark}
For $n=2$, the above theorem does not hold, not even if $S=S^n$ is the standard sphere.
However, if we replace the condition that $f$ be $\Lambda^2$-$1$-contracting by the stronger condition that $f$ is $1$-Lipschitz, i.e., $|df|_\op\le1$, then the above proof also works if $n=2$.
\end{remark}



\begin{bibdiv}
\begin{biblist}

\bib{AS}{article}{
   author={Atiyah, Michael F.},
   author={Singer, Isadore M.},
   title={The index of elliptic operators on compact manifolds},
   journal={Bull. Amer. Math. Soc.},
   volume={69},
   date={1963},
   number={3},
   pages={422--433},
   issn={0002-9904},
   zbl={0118.31203},
   mr={0157392},
   doi={10.1090/S0002-9904-1963-10957-X},
}

\bib{APS1}{article}{
   author={Atiyah, Michael F.},
   author={Patodi, Vijay K.},
   author={Singer, Isadore M.},
   title={Spectral asymmetry and Riemannian geometry. I},
   journal={Math. Proc. Cambridge Philos. Soc.},
   volume={77},
   date={1975},
   pages={43--69},
   issn={0305-0041},
   zbl={0297.58008},
   mr={0397797},
   doi={10.1017/S0305004100049410},
}

\bib{APS2}{article}{
   author={Atiyah, Michael F.},
   author={Patodi, Vijay K.},
   author={Singer, Isadore M.},
   title={Spectral asymmetry and Riemannian geometry. II},
   journal={Math. Proc. Cambridge Philos. Soc.},
   volume={78},
   date={1975},
   pages={405--432},
   issn={0305-0041},
   zbl={0314.58016},
   mr={0397798},
   doi={10.1017/S0305004100049410},
}

\bib{APS3}{article}{
   author={Atiyah, Michael F.},
   author={Patodi, Vijay K.},
   author={Singer, Isadore M.},
   title={Spectral asymmetry and Riemannian geometry. III},
   journal={Math. Proc. Cambridge Philos. Soc.},
   volume={79},
   date={1976},
   pages={71--99},
   issn={0305-0041},
   zbl={0325.58015},
   mr={0397799},
   doi={10.1017/S0305004100049410},
}


\bib{Baer3}{article}{
 author={Bär, Christian},
 issn={1016-443X},
 issn={1420-8970},
 doi={10.1007/BF02246994},
 zbl={0867.53037},
 mr={1421872},
 title={Metrics with harmonic spinors},
 journal={Geom. Func. Anal.},
 volume={6},
 number={6},
 pages={899--942},
 date={1996},
 publisher={Springer (Birkh{\"a}user), Basel},
 eprint={https://eudml.org/doc/58251},
}

\bib{Baer}{arxiv}{
 author={Bär, Christian},
 arx={2407.21704},
 title={Dirac eigenvalues and the hyperspherical radius},
 year={2024},
 note={to appear in J. Europ. Math. Soc.},
}

\bib{BB}{article}{
 author={Bär, Christian},
 author={Ballmann, Werner},
 isbn={978-1-57146-237-4},
 zbl={1331.58022},
 mr={3076058},
 doi={10.4310/SDG.2012.v17.n1.a1},
 title={Boundary value problems for elliptic differential operators of first order},
 journal={Surv. Differ. Geom.},
 volume={17},
 pages={1--78},
 date={2012},
}

\bib{Bott}{article}{
 author={Bott, Raoul},
 issn={0003-486X},
 issn={1939-8980},
 doi={10.2307/1970106},
 zbl={0129.15601},
 mr={0110104},
 title={The stable homotopy of the classical groups},
 journal={Ann. Math. (2)},
 volume={70},
 pages={313--337},
 date={1959},
 publisher={Princeton University, Mathematics Department, Princeton, NJ},
}

\bib{CZ}{article}{
 author={Cecchini, Simone},
 author={Zeidler, Rudolf},
 issn={1465-3060},
 issn={1364-0380},
 doi={10.2140/gt.2024.28.1167},
 zbl={1546.53040},
 mr={4746412},
 title={Scalar and mean curvature comparison via the Dirac operator},
 journal={Geom. \& Topol.},
 volume={28},
 number={3},
 pages={1167--1212},
 date={2024},
 publisher={Mathematical Sciences Publishers (MSP), Berkeley, CA; Geometry \& Topology Publications c/o University of Warwick, Mathematics Institute, Coventry},
}

\bib{CS}{article}{
 author={Chern, Shiing-Shen},
 author={Simons, James},
 issn={0003-486X},
 issn={1939-8980},
 doi={10.2307/1971013},
 zbl={0283.53036},
 mr={0353327},
 title={Characteristic forms and geometric invariants},
 journal={Ann. Math. (2)},
 volume={99},
 pages={48--69},
 date={1974},
 publisher={Princeton University, Mathematics Department, Princeton, NJ},
 eprint={semanticscholar.org/paper/2c58ec96da220d20b0891fdff59157f7cf3e3585},
}

\bib{SpectralFlowBook}{book}{
   author={Doll, Nora},
   author={Schulz-Baldes, Hermann},
   author={Waterstraat, Nils},
   title={Spectral flow. A functional analytic and index-theoretic approach},
   series={De Gruyter Studies in Mathematics},
   volume={94},
   publisher={De Gruyter, Berlin},
   date={2023},
   pages={viii+234},
   isbn={978-3-11-111626-1},
   isbn={978-3-11-111725-1},
   zbl={1545.58001},
   mr={4663058},
   doi={10.1515/9783111117270},
}

\bib{Ronge}{article}{
 author={van den Dungen, Koen},
 author={Ronge, Lennart},
 issn={1846-3886},
 issn={1848-9974},
 doi={10.7153/oam-2021-15-87},
 zbl={1512.47022},
 mr={4364605},
 title={The APS-index and the spectral flow},
 journal={Operators and Matrices},
 volume={15},
 number={4},
 pages={1393--1416},
 date={2021},
 publisher={ELEMENT, Zagreb},
}

\bib{Ebert}{article}{
 author={Ebert, Johannes},
 issn={0002-9947},
 issn={1088-6850},
 doi={10.1090/tran/7133},
 zbl={1373.19005},
 mr={3683115},
 title={The two definitions of the index difference},
 journal={Transactions of the American Mathematical Society},
 volume={369},
 number={10},
 pages={7469--7507},
 date={2017},
 publisher={American Mathematical Society (AMS), Providence, RI},
}

\bib{G}{article}{
   author={Getzler, Ezra},
   title={The odd Chern character in cyclic homology and spectral flow},
   journal={Topology},
   volume={32},
   date={1993},
   number={3},
   pages={489--507},
   issn={0040-9383},
   zbl={0801.46088},
   mr={1231957},
   doi={10.1016/0040-9383(93)90002-D},
}

\bib{Gilkey}{book}{
   author={Gilkey, Peter B.},
   isbn={0-8493-7874-4},
   isbn={0-914098-20-9},
   publisher={Boca Raton, FL: CRC Press},
   zbl={0856.58001},
   mr={1396308},
   title={Invariance theory, the heat equation and the Atiyah-Singer index theorem},
   pages={ix + 516},
   date={1995},
   doi={10.1201/9780203749791},
}

\bib{GS}{article}{
   author={Goette, Sebastian},
   author={Semmelmann, Uwe},
   title={Scalar curvature estimates for compact symmetric spaces},
   journal={Differential Geom. Appl.},
   volume={16},
   date={2002},
   number={1},
   pages={65--78},
   issn={0926-2245},
   zbl={1043.53030},
   mr={1877585},
   doi={10.1016/S0926-2245(01)00068-7},
}

\bib{LM}{book}{
 author={Lawson, H. Blaine},
 author={Michelsohn, Marie-Louise},
 isbn={0-691-08542-0},
 zbl={0688.57001},
 title={Spin geometry},
 mr={1031992},
 pages={xii + 427},
 date={1989},
 publisher={Princeton University Press, Princeton, NJ},
 url={https://www.jstor.org/stable/j.ctt1bpmb28},
}

\bib{LSW}{article}{
   author={Li, Yihan},
   author={Su, Guangxiang},
   author={Wang, Xiangsheng},
   title={Spectral flow, Llarull's rigidity theorem in odd dimensions and its generalization},
   journal={Sci. China Math.},
   volume={67},
   date={2024},
   number={5},
   pages={1103--1114},
   issn={1674-7283},
   mr={4739559},
   doi={10.1007/s11425-023-2138-5},
   zbl={1545.53039},
}

\bib{Ll}{article}{
   author={Llarull, Marcelo},
   title={Sharp estimates and the Dirac operator},
   journal={Math. Ann.},
   volume={310},
   date={1998},
   number={1},
   pages={55--71},
   issn={0025-5831},
   zbl={0895.53037},
   mr={1600027},
   doi={10.1007/s002080050136},
}

\bib{MiMo}{article}{
 author={Milnor, John W.},
 author={Moore, John C.},
 issn={0003-486X},
 issn={1939-8980},
 doi={10.2307/1970615},
 zbl={0163.28202},
 mr={0174052},
 title={On the structure of Hopf algebras},
 journal={Ann. Math. (2)},
 volume={81},
 pages={211--264},
 date={1965},
 publisher={Princeton University, Mathematics Department, Princeton, NJ},
 eprint={polipapers.upv.es/index.php/AGT/article/view/2250},
}

\bib{Phillips}{article}{
   author={Phillips, John},
   title={Self-adjoint Fredholm operators and spectral flow},
   journal={Canad. Math. Bull.},
   volume={39},
   date={1996},
   number={4},
   pages={460--467},
   issn={0008-4395},
   zbl={0878.19001},
   mr={1426691},
   doi={10.4153/CMB-1996-054-4},
}

\bib{Shi1}{article}{
 author={Shi, Pengshuai},
 issn={0001-8708},
 issn={1090-2082},
 doi={10.1016/j.aim.2025.110429},
 zbl={08091843},
 mr={4929482},
 title={Spectral flow of Callias operators, odd {{\(\mathrm{K}\)}}-cowaist, and positive scalar curvature},
 journal={Adv. Math.},
 volume={479},
 pages={41 pages},
 note={Id/No 110429},
 date={2025},
 publisher={Elsevier (Academic Press), San Diego, CA},
}

\bib{Shi2}{article}{
 author={Shi, Pengshuai},
 issn={1073-7928},
 issn={1687-0247},
 doi={10.1093/imrn/rnaf262},
 zbl={08087035},
 mr={4951381},
 title={The odd-dimensional long neck problem via spectral flow},
 journal={Intern. Math. Res. Notices},
 volume={2025},
 number={17},
 pages={19 pages},
 note={Id/No rnaf262},
 date={2025},
 publisher={Oxford University Press, Cary, NC},
}

\bib{YWX}{article}{
 author={Yu, Yue},
 author={Wu, Yong-Shi},
 author={Xie, Xincheng},
 issn={0550-3213},
 issn={1873-1562},
 doi={10.1016/j.nuclphysb.2017.01.018},
 zbl={1356.82039},
 mr={3611419},
 title={Bulk-edge correspondence, spectral flow and Atiyah-Patodi-Singer theorem for the {{\(\mathcal{Z}_2\)}}-invariant in topological insulators},
 journal={Nuclear Physics. B},
 volume={916},
 pages={550--566},
 date={2017},
 publisher={Elsevier (North-Holland), Amsterdam},
}

\end{biblist}
\end{bibdiv}
\end{document}